\documentclass[11pt]{amsproc}
\usepackage{amsfonts}
\usepackage{xy}
\usepackage{amsmath,amstext,amsbsy,amssymb}

\newtheorem{theorem}{Theorem}[section]

\newtheorem{proposition}[theorem]{Proposition}
\newtheorem{corollary}[theorem]{Corollary}

\theoremstyle{definition}
\newtheorem{definition}[theorem]{Definition}

\theoremstyle{remark}

\numberwithin{equation}{section}

\begin{document}
\title{Lie Superalgebras arising from bosonic representation}
\author{Naihuan Jing}
\address{School of Sciences, South China University of Technology, Guangzhou 510640, China
and Department of Mathematics, North Carolina State University, Raleigh, NC 27695, USA}
\email{jing@math.ncsu.edu}
\author{Chongbin Xu*}
\address{School of Science, South China University of Technology, Guangzhou 510640
and School of Mathematics \& Information, Wenzhou University,
Zhejiang 325035, China} \email{xuchongbin1977@126.com}
\thanks{*Corresponding author}\thanks{Jing gratefully acknowledges the support from
Simons Foundation and NSF}

\begin{abstract} A 2-toroidal Lie superalgebra is constructed using bosonic fields and
a ghost field. The superalgebra contains  $osp(1|2n)^{(1)}$ as a
distinguished subalgebra and behaves similarly to the toroidal Lie
superalgebra of type $B(0, n)$. Furthermore this algebra is a
central extension of the algebra $osp(1|2n)\otimes \mathbb C[s,
s^{-1}, t,t^{-1}]$.
\end{abstract}
\maketitle

\section{Introduction}
Realization of Lie (super)algebras via bosonic or fermionic fields
has been a successful approach in solving various problems in
mathematical physics. From mathematical perspectives, representing an
algebra in terms of known classical algebraic structures (Weyl or
Clifford algebras) amounts to constructing a homomorphism from the
initial algebra to the target algebra. In this way properties of the
initial algebra can be studied by techniques of linear algebra. On
one hand representations of the target algebra give rise to
representations of the initial algebra which then enable one to gain
information about eigenvalues of the operators or the expectation
values of various physical quantities. On the other hand physical
intuition usually suggests how this homomorphism can be established,
and this may happen in several different contexts or via different
realizations. Mathematically it may be proved that all these
realizations are actually representing the same algebra.

 Lie (super)algebras play an important role in both mathematics and physics.  The realization problem,
   in particular for infinite dimensional cases,
 is one of the first questions to be studied. It was well-known in physics literature that finite-dimensional
 simple Lie (super)algebras of classical types
 can be realized by ferminoic or bosonic operators, i.e. within Clifford or Weyl algebras. However, the problem
 for affine Lie (super)algebras
 is more involved and requires sophisticated generalization. By 1980's the problem of realizing affine Lie
 (super)algebras had been solved partly
 by several groups \cite{F, KP}. In the famous work \cite{FF}, A. Feingold and I. Frenkel realized
 classical affine Lie algebras using ferminoic fields
 and bosonic fields respectively. Since then, their method has been generalized to other algebras such
 as extended affine Lie algebras,
 affine Lie superalgebras, Tits-Kantor-K\"ocher algebras, toroidal Lie algebras, Lie algebras with central extentions, two-parameter quantum affine algebras etc.
 (see \cite{G, KW, T, JMT, L, JZ}).

More recently, in \cite{JM, JMX} 2-toroidal Lie algebras of
classical types were realized uniformly using bosonic fields or
ferminoic fields with help of a ghost field based on the
Moody-Rao-Yokonuma presentation of toroidal Lie algebras \cite{MRY}.
The method used in these two papers not only generalizes
Feingold-Frenkel construction to toroidal algebras but also fills up
missing bosonic/fermionic realizations for orthogonal/symplectic
types.

One can ask a similar question on how to realize 2-toroidal Lie
superalgebra of classical types. 
In this paper, we first define a central extension of the
superalgebra $\mathfrak g\otimes \mathbb C[s, s^{-1}, t,t^{-1}]$
which we will call the loop-like toroidal Lie superalgebra of type
$B(0,n)$ in light of MRY presentation, and then construct its
representation using mixed bosons and fermions as well as a ghost
field. Since the kernel of the homomorphism is contained in fields
corresponding to imaginary roots, our representation can be lifted
to that of the universal central extension of the 2-toroidal Lie
superalgebra. It would be interesting to show that the loop-like
toroidal superalgebra is indeed the universal central extension of
the superalgebra $\mathfrak g\otimes \mathbb C[s, s^{-1},
t,t^{-1}]$. Our result may be viewed as a testing example for this
and we hope further computations can be made to reveal the structure
of the kernel of the map from the loop-like toroidal superalgebra to the universal central
extension.

\section{Toroidal Lie superalgebra  of type $B(0,n)$}
A Lie superalgebra $\mathfrak{g}=\mathfrak{g}_{\overline{0}}\oplus\mathfrak{g}_{\overline{1}}$ is
 a $\mathbb{Z}_{2}$-graded vector space
equipped with a bilinear map $[\cdot,\cdot]:\mathfrak{g}\times\mathfrak{g}\rightarrow\mathfrak{g}$ such that

   $\quad1)~ [\mathfrak{g}_{\alpha},\mathfrak{g}_{\beta}]\subseteq\mathfrak{g}_{\alpha+\beta},$

   $\quad2) ~[a,b]=-(-1)^{\mbox{deg}(a)\mbox{deg}(b)}[b,a],$

   $\quad3)~[a,[b,c]]=[[a,b],c]]+(-1)^{\mbox{deg}(a)\mbox{deg}(b)}[b,[a,c]],$ \\
   where $\alpha,\beta\in\mathbb{Z}_{2}$ and $a,b,c$ are homogenous elements.

According to Kac \cite{K1} simple Lie superalgebras are classified into two
families--the classical types and the Cartan
types. Among the classical superalgebras, one usually
separates the strange series $P(n)$ and $Q(n)$ from the list of basic
superalgebras: the series $A (m, n)$, 
$B(m,n), C(n+1)$, and $D(m,n)$ and the
exceptional types $F(4)$, $G(3)$ and $D(2, 1;
\alpha)$. The othorsymplectic series $B(m,n)$ can be further
divided into two classes: $m>0$ and $m=0$. In this paper we consider the
simplest case $B(0, n)$.

Let $R=\mathbb C[s^{\pm1},t^{\pm1}]$ be the ring of Laurant
polynomials in $s, t$. Let $\Omega_R$ be the $R$-module of
differentials spanned by $da, a\in R$, and $d\Omega_R$ is the space
of exact forms. Then $\Omega_R/d\Omega_R$ has a basis consisting of
$\overline{s^{p-1}t^{q}ds}$, $\overline{s^{p}t^{-1}dt}$,
$\overline{s^{-1}ds}$.  Let $\mathfrak g$ be a simple Lie
superalgebra, the toroidal Lie superalgebra $T(\mathfrak g)$ is the
central extension of the loop superalgebra $\mathfrak g\otimes R$:
\begin{equation*}
T({\mathfrak g})={\mathfrak g}\otimes R\oplus \Omega_R/d\Omega_R
\end{equation*}
under the Lie (super)bracket: ($x,y\in \mathfrak{g},~a,b\in R$)
$$[x\otimes a, y\otimes b]=[x, y]\otimes ab+(x|y)\overline{(da)b}.$$
with parities defined by
$$\mbox{deg}(x\otimes a)=\mbox{deg}(x),\quad \mbox{deg}(\Omega_R/d\Omega_R)=\overline{0}$$

In practice, one can define a Lie superalgebra by generators and
relations with appropriate
   parities for generators.
   In what follows, we will define the so-called toroidal Lie superalgebra of type $B(0,n)$ this way.

For $n\geq1$, let $A=(a_{ij})$ be the extended distinguished Cartan
matrix of the affine Lie superalgebra $B(0,n)^{(1)}$, i.e.
\begin{equation*}
A=\begin{pmatrix} 2 & -1 & 0 &\cdots & 0& 0 & 0 \\
-2 & 2 & -1& \cdots & 0& 0 & 0 \\
0 & -1 & 2& \cdots & 0& 0 & 0 \\
\cdot & \cdot & \cdot & \cdots & \cdot & \cdot &\cdot \\
0 & 0 & 0& \cdots & 2&-1 & 0 \\
0 & 0 & 0& \cdots & -1 &2 & -1 \\
0 & 0 & 0& \cdots &0&-2 & 2
\end{pmatrix},
\end{equation*}
and let
$Q=\mathbb{Z}\alpha_{0}\oplus\mathbb{Z}\alpha_{1}\oplus\cdots\oplus\mathbb{Z}\alpha_{n}$
be the root lattice, where $\alpha_0, \cdots, \alpha_{n-1}$ are even
roots and $\alpha_n$ is odd. The standard invariant form is given by
\begin{equation*}
(\alpha_i, \alpha_j)=d_ia_{ij},
\end{equation*}
where $(d_0, d_1, \cdots, d_{n-1}, d_n)=(2, 1, \cdots, 1, 1/2)$,
thus $2(\alpha_i, \alpha_j)/(\alpha_i, \alpha_i)=a_{ij}$. Let
$\delta=\alpha_0+2\sum_{i=1}^{n}\alpha_i$.

\begin{definition}\label{def} The (loop-like) toroidal Lie superalgebra $\mathfrak{T}$ of type $B(0,n)$ is the Lie superalgebra generated by
$$\{\mathcal{K},\alpha_{i}(k),x^{\pm}_{i}(k)|\, 0\leq i\leq n,k\in\mathbb{Z}\}$$
with parities given as :
$$\mbox{deg}(\mathcal{K})=\mbox{deg}(\alpha_{i}(k))=\overline{0}, \quad 0\leq i\leq n,k\in\mathbb{Z};$$
$$~~\mbox{deg}(x^{\pm}_{i}(k))=
\left\{
\begin{array}{ll}
\overline{0}&\mbox{if}~0\leq i\leq n-1,k\in\mathbb{Z};\\
\overline{1}&\mbox{if}~i=n,k\in\mathbb{Z}.
\end{array}
\right.$$
and defining relations:
\begin{eqnarray*}
   &1)&[\mathcal{K},\alpha_{i}(k)]=[\mathcal{K},x^{\pm}_{i}(k)]=0; \\
   &2)& [\alpha_{i}(k),\alpha_{j}(l)]=k(\alpha_{i}|\alpha_{j})\delta_{m,-n}\mathcal{K}; \\
   &3)& [\alpha_{i}(k),x^{\pm}_{j}(l)]=\pm(\alpha_{i}|\alpha_{j})x^{\pm}_{j}(k+l); \\
   &4)& [x^{+}_{i}(k),x^{-}_{j}(l)]=-\delta_{ij}\frac{2}{(\alpha_{i}|\alpha_{j})}\{\alpha_{i}(k+l)
   +k\delta_{k,-l}\mathcal{K}\};\\
   &5)&[x^{+}_{i}(k),x^{+}_{j}(l)]=[x^{-}_{i}(k),x^{-}_{j}(l)]=0;\\
   &&(\mbox{ad}x^{+}_{i}(0))^{1-a_{ij}}x^{+}_{j}(k))=0,\mbox{if}~i\neq j;\\
   &&(\mbox{ad}x^{-}_{i}(0))^{1-a_{ij}}x^{-}_{j}(k))=0,\mbox{if}~i\neq j,
\end{eqnarray*}
where the super-bracket is adopted: $[X,Y]=XY-(-1)^{\mbox{deg}(X)\mbox{deg}(Y)}YX$.
\end{definition}

The algebra $\mathfrak{T}$ is a $Q\times \mathbb Z$-graded Lie
superalgebra under the grading: $deg \mathcal{K}=(0, 0), deg
\alpha_i(k)=(0, k), deg\, x^{\pm}_i(k)=(\pm\alpha_i, k)$. We denote
the subspace of degree $(\alpha, k)$ by $\mathfrak T_k^{\alpha}$:
\begin{equation}
\displaystyle \mathfrak T=\oplus_{(\alpha, k)\in Q\times \mathbb
Z}\mathfrak T_k^{\alpha}.
\end{equation}
We remark that the center of this algebra is contained in the
subalgebra generated by $\mathfrak T_k^{n\delta}$, $k, n\in \mathbb
Z$.

Let $e_{i},f_{i},h_{i}$ ($i=1,2\cdots,n$) be the Chevalley
generators of Lie superalgebra $\mathcal {G}$ of type $B(0,n)$
corresponding to the distinguished simple root system
$\Pi=\{\alpha_{1},\cdots,\alpha_{n}\}$ and $\theta$ be the longest
root relative to $\Pi$. We also choose $e_{0}\in\mathcal
{G}_{\theta},f_{0}\in\mathcal {G}_{-\theta}$ and $h_{0}$ as in the
affine Lie algebra, then we have

\begin{proposition} \label{homo} The following map defines a surjective
homomorphism from loop-like toroidal superalgebra $\mathfrak T$ to
the algebra $T(B(0, n))$:
\begin{align*}
\mathcal{K}&\mapsto \overline{s^{-1}ds}\\
\alpha_i(k)&\mapsto d_i(h_i\otimes s^k+\delta_{i0}\overline{s^kt^{-1}dt}), \qquad i=0, \cdots, n\\
x^+_i(k)&\mapsto e_i\otimes s^k, \qquad i=1, \cdots, n\\
x^-_i(k)&\mapsto -f_i\otimes s^k, \qquad i=1, \cdots, n\\
x^+_0(k)&\mapsto e_0\otimes s^kt^{-1}, \\
x^-_0(k)&\mapsto -f_0\otimes s^kt
\end{align*}
\end{proposition}
\begin{proof} It is straightforward to check that the elements on the right satisfy the relations
in Definition \ref{def}.
For example,
\begin{align*}
[h_i\otimes s^k, h_j\otimes s^l]&=[h_i, h_j]\otimes s^{k+l}+(h_i, h_j)\overline{s^kd(s^l)}\\
&=d_i^{-1}d_j^{-1}k(\alpha_i, \alpha_j)\delta_{k,-l}\overline{s^{-1}ds}.\\
\end{align*}
\end{proof}

In fact the loop-like toroidal Lie superalgebra is a central
extension of the algebra $\mathfrak g\otimes \mathbb C[t, t^{-1}, s,
s^{-1}]$, as the kernel is contained in the subspace $\oplus_{n,k}
\mathfrak T^{n\delta}_k$, which is clearly central by the
commutation relations. General theory of central extension of Lie
superalgebra has been studied as the usual Lie algebras \cite{IK}.
It would be interesting to show that the algebra $\mathfrak T$ is
indeed the universal central extension.

For convenience we will present the structure of $\mathfrak{T}$ in terms of generating series. To this end, we define the generating series with coefficients from $\mathfrak{T}$:
$$\alpha_{i}(z)=\sum_{k\in\mathbb{Z}}\alpha_{i}(k)z^{-k-1},\quad x^{\pm}_{i}(z)=\sum_{k\in\mathbb{Z}}x^{\pm}_{i}(k)z^{-k-1}.$$

\bigskip

\begin{proposition}\emph{The relations of $\mathfrak{T}$ can be written as
follows.}
 \begin{eqnarray*}
    &1')& [\mathcal{K},\alpha_{i}(z)]=[\mathcal{K},x^{\pm}_{i}(z)]=0; \\
    &2')& [\alpha_{i}(z),\alpha_{j}(w)]=(\alpha_{i}|\alpha_{j})\partial_{w}\delta(z-w)\mathcal{K};\\
    &3')& [\alpha_{i}(z),x^{\pm}_{j}(w)]=\pm(\alpha_{i}|\alpha_{j})x^{\pm}_{j}(w)\delta(z-w); \\
    &4')& [x^{+}_{i}(z),x^{-}_{j}(w)]=-\delta_{ij}\frac{2}{(\alpha_{i}|\alpha_{j})}\{(\alpha_{i}(w)\delta(z-w)+\partial_{w}\delta(z-w)\mathcal{K}\};\\
    &5')& [x^{+}_{i}(z),x^{+}_{i}(w)]=[x^{-}_{i}(z),x^{-}_{i}(w)]=0;\\
    &&  \big(\prod_{k=1}^{1-a_{ij}}\mbox{ad}x^{+}_{i}(z_{k})\big)x^{+}_{j}(w)=0,\mbox{if}~i\neq j;\\
    &&  \big(\prod_{k=1}^{1-a_{ij}}\mbox{ad}x^{-}_{i}(z_{k})\big)x^{-}_{j}(w)=0,\mbox{if}~i\neq j.
 \end{eqnarray*}
\end{proposition}
\begin{proof} The proposition follows from definition 2.1 and the
two useful expansions of formal delta function
$\delta(z-w)=\sum_{k\in\mathbb{Z}}w^{k}z^{-k-1}$ (cf. \cite{K2}):
$$\delta(z-w)=i_{z,w}\frac{1}{z-w}+i_{w,z}\frac{1}{z-w},$$
$$ \partial_{w}\delta(z-w)=i_{z,w}\frac{1}{(z-w)^{2}}-i_{w,z}\frac{1}{(z-w)^{2}},$$
where the symbol $i_{z,w}$ means expansion in the domain $|z|>|w|$.
For simplicity $i_{z,w}$ is omitted when there is no confusion in expansion direction.
\end{proof}

\section{Representation of $\mathfrak{T}$}
In this section, we realize the loop-like toroidal Lie superalgebra defined in section 2
using bosonic fields and a ghost field.

Let $\{\varepsilon_{i}|0\leq i\leq n+1\}$ be an orthonormal basis of
the vector space $\mathbb{C}^{n+2}$ with inner product
$(\varepsilon_{i},\varepsilon_{j})=\delta_{ij}$. Then the
distinguished simple and positive roots of simple Lie superalgebra
of type $B(0,n)$ can be realized as follows:
$$\Pi:\alpha_{1}=\varepsilon_{1}-\varepsilon_{2},\cdots,\alpha_{n-1}=\varepsilon_{n-1}-\varepsilon_{n},
\alpha_{n}=\varepsilon_{n}$$
$$\Delta_{+}=\big\{2\varepsilon_{i},\varepsilon_{i}\pm\varepsilon_{j},\varepsilon_{i}|1\leq i\leq n,1\leq i<j\leq n\big\}$$
and the longest root is
$$\theta=2\alpha_{1}+\cdots+2\alpha_{n}=2\varepsilon_{1}$$

Introduce $\overline{c}=\varepsilon_{0}+\sqrt{-1}\varepsilon_{n+1}$
and define
$$\alpha_{0}=\overline{c}-\theta~\mbox{and}~\beta=\varepsilon_{1}-\frac12\overline{c}$$
then
$$\mbox{P}=\mathbb{Z}\overline{c}\oplus\mathbb{Z}\varepsilon_{1}\oplus\cdots\oplus\mathbb{Z}\varepsilon_{n}$$
is the weight lattice of $B(0,n)^{(1)}$. Note that
$(\beta|\beta)=1,(\beta|\varepsilon_{i})=\delta_{1i}$,
$\alpha_0=-2\beta$.

\bigskip

Let $\mbox{P}_{\mathbb{C}}=\mbox{P}\otimes_{\mathbb{Z}}\mathbb{C}$ be the vector space spanned by $\overline{c},\varepsilon_{i},1\leq i\leq n$
over $\mathbb{C}$ and define $\mathcal{C}=\mathcal{C}_{0}\oplus\mathcal{C}_{1}$,where $\mathcal{C}_{0}=\mbox{P}_{\mathbb{C}},\mathcal{C}_{1}=\mbox{P}^{*}_{\mathbb{C}}$ are both maximal isotrophic subspaces under the antisymmetric bilinear form on $\mathcal{C}$ given by
$$\langle b^{*},a\rangle=-\langle a,b^{*}\rangle=(a,b),~\langle a,b\rangle=\langle a^{*},b^{*}\rangle=0,a,b\in \mathcal{C}_{0}$$

Let $\Lambda(\infty)$ be the associative algebra generated by
$$\big\{u(k),e(k)|u\in\mathcal{C},k\in\mathbb{Z}\big\}$$
with the defining relations:
$$u(k)v(l)-v(l)u(k)=\langle u,v\rangle\delta_{k,-l}, $$
and
$$u(k)e(l)-e(l)u(k)=0,\quad e(k)e(l)+e(l)e(k)=\delta_{k,-l},$$
where $u,v\in \mathcal{C},k,l\in\mathbb{Z}$.

\bigskip

We define the normal ordering of a quadratic expression to be
$$
:u(k)v(l):=\left\{
\begin{array}{lll}
u(k)v(l),&k<0\\
\frac{1}{2}(u(0)v(l)+v(l)u(0)),&k=0\\
v(l)u(k),&k>0
\end{array}
\right.
$$
$$
:e(k)e(l):=\left\{
\begin{array}{lll}
e(k)e(l),&k<0\\
\frac{1}{2}(e(0)e(l)-e(l)e(0)),&k=0\\
-e(l)e(k),&k>0
\end{array}
\right.
$$
and $$:u(k)e(l):=:e(l)u(k):=u(k)e(l),k,l\in\mathbb{Z}.$$

\bigskip

Introduce the $\Lambda(\infty)$-module

$$V=\Lambda(\infty)\Big/\Big[\sum_{k\in\mathbb{Z}^{+}}\Big(\Lambda(\infty)u(k)+\Lambda(\infty)e(k)\Big)\Big],$$
which is isomorphic to the associative algebra $\Lambda^{-}(\infty)$
generated by $u(k),e(k),u\in\mathcal{C},k\in\mathbb{Z}_{-}$ as
vector spaces. Denote the image of $x$ by $x|0\rangle$, then we have
$$V=\Lambda^{-}(\infty)|0\rangle.$$
It is clear that the following power series
$$u(z)=\sum_{k\in\mathbb{Z}}u(k)z^{-k-1/2},e(z)=\sum_{k\in\mathbb{Z}}e(k)z^{-k-1/2}$$
are bosonic and ferminoic fields, respectively.

For any two bosonic(ferminoic) fields, we let
$$:x(z)y(w):=\sum_{k,l\in\mathbb{Z}}:x(k)y(l):z^{-k-1/2}w^{-l-1/2}.$$
Then it follows that
$$:u(z)v(w):=:v(w)u(z):~\mbox{and}~:e(z)e(w):=-:e(w)e(z):$$
and
$$:u(z)e(w):=u(z)e(w)=:e(w)u(z):$$
Based on the normal ordering of two fields, one can define the normal product of $n$ fields inductively.

\bigskip

We further define the contraction of two fields by
$$\underbrace{x(z)y(w)}=x(z)y(w)-:x(z)y(w):$$

Since the element $c$ is central, we have that for any $u\in\mathcal C$
\begin{equation}\label{ghost}
[c(z), u(w)]=0.
\end{equation}

\begin{definition} Let $x_{1},\cdots,x_{n}$ be generators in an algebra of
operators having bosonic relations and a notions of normal ordering.
Define
$$:x_{1}\cdots\underbrace{x_i\cdots x_j}\cdots x_{n}:=\underbrace{x_i x_j}(:x_{1}\cdots x_{i-1}x_{i+1}\cdots x_{j-1}x_{j+1}\cdots x_{n}:)$$
\end{definition}
\textbf{Remark:} In what follows, we understand that $\langle e,e\rangle=1,\langle e,\mathcal{C}\rangle=0$,
so the above definition is also well defined when one or two of $x_{i}$'s is $e$.

\bigskip

The following Wick's theorem is well-known.

\begin{theorem} (\cite{FF, K2}) \emph{ For elements $x_{i},y_{j},1\leq i\leq
r,1\leq j\leq t,$ we have
$$:x_{1}\cdots x_{r}::y_{1}\cdots y_{s}:=:x_{1}\cdots x_{r}y_{1}\cdots y_{s}:
+\sum \pm :x_{1}\underbrace{\cdots x_{r}y_{1}\cdots }y_{s}:$$ where the summation is taken over all
 possible combinations of contraction of some $x$'s and some $y$'s,
 and the sign is the sign of permutation of fermionic operators.}
\end{theorem}

\medskip
\begin{corollary}\emph{In an algebra with both fermionic and bosonic generators, we have for any permutation
$$:x_{1}\cdots x_{r}:=(-1)^{N}:x_{\sigma(1)}\cdots x_{\sigma(r)}:$$
where $N$ is the number of fermionic-fermionic transposition in a decomposition of $\sigma$.}
\end{corollary}

\medskip
The following statement is proved by the standard technique of OPE
(cf. \cite{JMX}). The only difference occurs when one or two factors
are odd operators, but this is taken care of by super-brackets from
the construction of the field operators.

\begin{corollary}\emph{For $r_{1},s_{1}\in \mathcal{C},r_{2},s_{2}\in \mathcal{C}\cup\{e\}$ and $|z|>|w|$,we have}
\begin{eqnarray*}
&&\big(:r_{1}(z)s_{1}(z):\big)\big(:r_{2}(w)s_{2}(w):\big)\\
   &=& :r_{1}(z)s_{1}(z)r_{2}(w)s_{2}(w):+\Big(\langle r_{1},r_{2}\rangle:s_{1}(z)s_{2}(w):+\langle r_{1},s_{2}\rangle \\
   &&\cdot:s_{1}(z)r_{2}(w):+\langle s_{1},r_{2}\rangle:r_{1}(z)s_{2}(w):+\langle s_{1},s_{2}\rangle:r_{1}(z)r_{2}(w): \Big)  \\
  &&\cdot\frac{1}{z-w}+\Big(\langle r_{1},r_{2}\rangle\langle s_{1},s_{2}\rangle+\langle r_{1},s_{2}\rangle\langle s_{1},
  r_{2}\rangle\Big)\cdot\frac{1}{(z-w)^{2}}
\end{eqnarray*}
\end{corollary}
\medskip
The following result is an immediate consequence.
\begin{proposition}\emph{The bosonic fields satisfy the following (super)commutation relations:
$$[a(z),b(w)]=[a^{*}(z),b^{*}(w)]=0$$
$$[a(z),b^{*}(w)]=\langle a,b\rangle\delta(z-w)$$
and the commutators among normal ordering products are given by
$$[:a_{1}(z)a_{2}^*(z):,:b_{1}(w)b_{2}^*(w):]=-(a_{1},b_{2}):a_{2}^{*}(z)b_{1}(z):\delta(z-w)$$
$$\qquad\qquad\qquad+(a_{2},b_{1}):a_{1}(z)b_{2}^{*}(z):\delta(z-w)-(a_{1},b_{2})(a_{2},b_{1})\partial_{w}\delta(z-w),$$
 where $a,b,a_{i},b_{i}\in P_{\mathbb{C}},i=1,2$.}
\end{proposition}
\bigskip

The anti-symmetric inner product of the underlying Lie superalgebra
can be extended to an inner product for the span of quadratic
products:
$$\langle:r_{1}r_{2}:,:s_{1}s_{2}:\rangle=-\langle r_{1},s_{1}\rangle\langle r_{2},s_{2}\rangle+\langle r_{1},s_{2}\rangle\langle r_{2},s_{1}\rangle .$$
Now we can state and prove the main theorem.

\begin{theorem}\emph{The following correspondence
$$
x_{i}^{+}(z)=\left\{
\begin{array}{lll}
\frac{1}{2}:\beta^*(z)\beta^{*}(z):&i=0\\
:\varepsilon_{i}(z)\varepsilon_{i+1}^{*}(z):&1\leqslant i\leqslant n-1.\\
\sqrt 2:\varepsilon_{n}(z)e(z):&i=n
\end{array}
\right.\qquad\qquad
$$
$$
x_{i}^{-}(z)=\left\{
\begin{array}{lll}
\frac{1}{2}:\beta(z)\beta(z):&i=0\\
-:\varepsilon_{i+1}(z)\varepsilon_{i}^{*}(z):&1\leqslant i\leqslant n-1.\\
-\sqrt 2:\varepsilon^{*}_{n}(z)e(z):&i=n
\end{array}
\right.\qquad\qquad
$$
$$\qquad\qquad
\alpha_{i}(z)=\left\{
\begin{array}{llll}
-2:\beta^{*}(z)\beta(z):& i=0\\     
:\varepsilon_{i}(z)\varepsilon_{i}^{*}(z):-:\varepsilon_{i+1}(z)          
\varepsilon_{i+1}^{*}(z):&1\leqslant i\leqslant n-1.\\
:\varepsilon_{n}(z)\varepsilon_{n}^{*}(z):&i=n                
\end{array}
\right.
$$
gives rise to a realization of toroidal Lie superalgebra $\mathfrak{T}$ of level $-1$. 
Moreover, the correspondence also gives a representation of
2-toroidal superalgebra $T(B(0, n))$ through the map in Prop.
(\ref{homo}).}
\end{theorem}
\medskip

\begin{proof} First of all, we have for $i\neq j$
\begin{align*}
&[:\varepsilon_i(z)\varepsilon_j^*(z):,
   :\varepsilon_j(w)\varepsilon_i^*(w):]\\
   &\qquad =\big(:\varepsilon_i(z)\varepsilon_i^*(z):-:\varepsilon_j(z)
   \varepsilon_j^*(z):\big)\delta(z-w)-\partial_w\delta(z-w),\\
&[:\varepsilon_i^*(z)\varepsilon_j^*(z):,
:\varepsilon_i(w)\varepsilon_j(w):]\\
&\qquad=\big(:\varepsilon_j^*(z)\varepsilon_j(z):+:\varepsilon_i^*(z)
\varepsilon_i^*(z):\big)\delta(z-w)+\partial_w\delta(z-w),
\end{align*}
which imply that all commutation relations for $X(\alpha_i, z), i=1,
\cdots, n-1$ are the same as those of the affine Lie algebra of type
$A_{n-1}$ (cf. \cite{JMX}). The new commutation relations for the
Heisenberg algebra are
\begin{align*}
  [\alpha_{n-1}(z),\alpha_{n}(w)] &= -[:\varepsilon_{n}(z)    
\varepsilon_{n}^{*}(z):,:\varepsilon_{n}(w)
\varepsilon_{n}^{*}(w):] \\
 &=
 \partial_{w}\delta(z-w)=(\alpha_{n-1},\alpha_{n})\partial_{w}\delta(z-w)\cdot(-1), \\ 
[\alpha_{n}(z),\alpha_{n}(w)]&=[:\varepsilon_{n}(z)
\varepsilon_{n}^*(z):,:\varepsilon_{n}(w)
\varepsilon_{n}^*(z):]\\
&=(\alpha_{n},\alpha_{n})\partial_{w}\delta(z-w)\cdot(-1).
\end{align*}
Using Eq. (\ref{ghost}) we have that
\begin{align*}
&[:\varepsilon_{1}(z) \varepsilon_{1}^*(z):,:\beta(w)
\beta^*(w):]\\=&[:\varepsilon_{1}(z) \varepsilon_{1}^*(z):,:\epsilon_1(w)\epsilon_1(w):]\\
=&-\partial_{w}\delta(z-w).
\end{align*}
It follows that
\begin{align*}
[\alpha_0(z),
\alpha_0(w)]&=4[:\beta(z)\beta^{*}(z):,:\beta(w)\beta^{*}(w):]\\
&=-4\partial_{w}\delta(z-w)=(\alpha_0,
\alpha_0)\partial_{w}\delta(z-w)\cdot(-1).\\
[\alpha_0(z), \alpha_1(w)]&=-2[:\beta(z)\beta^{*}(z):,
:\varepsilon_{1}(z)\varepsilon_{1}^{*}(z):]\\
&=(\alpha_0, \alpha_1)\partial_{w}\delta(z-w)\cdot(-1).
\end{align*}
Now we check the commutation relations involving root vectors.
\begin{align*}
   [\alpha_{n-1}(z),x_{n}^{+}(w)]&=-\sqrt2[:\varepsilon_n(z)\varepsilon_n^*(z):, :\varepsilon_n(w)e(w)]\\
   &=-\sqrt2:\varepsilon_{n}(w)e(w):\delta(z-w)\\&=(\alpha_{n-1},\alpha_{n})x_{n}^{+}(w)\delta(z-w).\\
[\alpha_{n-1}(z),x_{n}^{-}(w)]&=-\sqrt2[:\varepsilon_n(z)\varepsilon_n^*(z):,
:\varepsilon_n^*(w)e(w)]\\&=\sqrt2:\varepsilon_n^*(z)e(w):\delta(z-w)\\
&=-(\alpha_{n-1},\alpha_{n})x_{n}^{-}(w)\delta(z-w).
\end{align*}
Similarly, we have
\begin{eqnarray*}
 [\alpha_{n}(z),x_{n}^{\pm}(w)]=\pm(\alpha_{n},\alpha_{n})x_{n}^{\pm}(w)\delta(z-w).
\end{eqnarray*}
\begin{eqnarray*}
  [x_{n}^{+}(z),x_{n}^{-}(w)]_{+} &=& -2[:\varepsilon_{n}(z)e(z):,:\varepsilon_{n}^*(w)e(w):]_+\\
   &=& -2\big\{:\varepsilon_{n}(w)\varepsilon_{n}^{*}(w):\delta(z-w)-\partial_{w}\delta(z-w)\big\} \\
   &=&
   \frac{-2}{(\alpha_{n},\alpha_{n})}\big\{\alpha_{n}(z)+\partial_{w}\delta(z-w)\cdot(-1)\big\},
  \end{eqnarray*}
where we have used the facts that
\begin{eqnarray*}
  :\varepsilon_{n}(z)e(z)\varepsilon_{n}(w)e(w):=-:\varepsilon_{n}(w)e(w)\varepsilon_{n}(z)e(z):
\end{eqnarray*}
 and $:e(w)e(w):=0$.
\begin{align*}
  [x_{0}^{+}(z),x_{0}^{-}(w)]&=\frac14[:\beta^*(z)\beta^*(z):,:\beta(w)\beta(w):]\\
   &=\frac14\{4:\beta(z)\beta^*(w):\delta(z-w)+2\partial_{w}\delta(z-w)\} \\
      &=
   \frac{-2}{(\alpha_{0},\alpha_{0})}\big\{\alpha_{0}(z)\delta(z-w)+\partial_{w}\delta(z-w)\cdot(-1)\big\},
  \end{align*}
\begin{align*}
  [\alpha_{0}(z),x_{0}^{+}(w)]&=-[:\beta(z)\beta^*(z):,:\beta^*(w)\beta^*(w):]\\
   &=2:\beta(z)\beta^*(w):\delta(z-w)=(\alpha_{0},\alpha_{0})x_{0}^{+}(w)\delta(z-w)
        \end{align*}
   similarly,
   $[\alpha_{0}(z),x_{0}^{-}(w)]=-(\alpha_{0},\alpha_{0})x_{0}^{-}(w)\delta(z-w)$.

Next we proceed to check Serre relations. In fact, we have that
$$[x_{i}^{\pm}(z),x_{n}^{\pm}(w)]=0,\quad0\leqslant i\leqslant n-2,$$ and
            \begin{eqnarray*}
            &&[x_{n-1}^{+}(z_{2}),[x_{n-1}^{+}(z_{1}),x_{n}^{+}(w)]] \\
            &=& \sqrt{2}[x_{n-1}^{+}(z_{2}),[:\varepsilon_{n-1}(z_{1})\varepsilon_{n}^{*}(z_{1}):,:\varepsilon_{n}(w)e(w):]] \\
            &=& \sqrt{2}[:\varepsilon_{n-1}(z_{2})\varepsilon_{n}^{*}(z_{2}):,:\varepsilon_{n-1}(w)e(w):]\delta(z_{1}/w)\\
            &=&0,
         \end{eqnarray*}
          \begin{eqnarray*}
             &&[x_{n}^{+}(z_{3}),[x_{n}^{+}(z_{2}),[x_{n}^{+}(z_{1}),x_{n-1}^{+}(w)]]]\\
             &=& \sqrt{2}[x_{n}^{+}(z_{3}),[x_{n}^{+}(z_{2}),[:\varepsilon_{n}(z_{1})e(z_{1}):,:\varepsilon_{n-1}(w)\varepsilon_{n}^{*}(w):]]] \\
             &=& -2[x_{n}^{+}(z_{3}),[:\varepsilon_{n}(z_{2})e(z_{2}):,:\varepsilon_{n-1}(w)e(w):]]\delta(z_{1}/w) \\
             &=& -2\sqrt{2}[:\varepsilon_{n}(z_{3})e(z_{3}):,:\varepsilon_{n}(w)\varepsilon_{n-1}(w):]\delta(z_{1}/w)\delta(z_{2}/w) \\
             &=& 0.
          \end{eqnarray*}
            $$\mbox{Similarly}\qquad\qquad [x_{n-1}^{-}(z_{2}),[x_{n-1}^{-}(z_{1}),x_{n}^{-}(w)]]=0, \qquad \qquad\qquad \qquad\qquad $$

$$~~[x_{n}^{-}(z_{3}),[x_{n}^{-}(z_{2}),[x_{n}^{-}(z_{1}),x_{n-1}^{-}(w)]]]=0. $$
Finally this realization indeed gives a representation of the
toroidal Lie superalgebra. In fact the imaginary field $\delta(z)$
clearly commute with all operators $x_{i}^{\pm}(z)$, and the kernel
of the map in Prop. (\ref{homo}) is contained in the subalgebra
generated by $\mathfrak T_k^{n\delta}$, our representation can be
lifted to the superalgebr $T(B(0, n))$. This completes the proof.
\end{proof}


\begin{thebibliography}{ABC}
\bibitem[F]{F} I. B. Frenkel, {\em Spinor representations of affine Lie algebras}, Proc. Natl. Acad. Sci.\\USA,77, No.11(1980), 6303-6306.
\bibitem[FF]{FF} A. J. Feingold,  I. B. Frenkel, {\em Classical affine algebras}, Adv. Math. 56 (1985), 117-172.
\bibitem[G]{G} Y. Gao, {\em Ferminoic and bosonic representation of the extended affine Lie algebra
of $\mathfrak{gl}_{N}(\mathbb{C}_{q})$}, Canad.  Math. Bull.  45 (2002), 623-633.
\bibitem[IK]{IK} K. Iohara, Y. Koga, {\em Central extensions of Lie
superalgebras}, Commun. Math. Helv. 76 (2001) 110--154.
\bibitem[JM]{JM} N. Jing, K. C. Misra, {\em Fermionic realization of toroidal Lie algebras of classical types}, J. Alg. 324 (2010), 183--194.
\bibitem[JMT]{JMT} N. Jing, K. C. Misra, S. Tan, {\em Bosonic realizations of higher level toroidal Lie algebras}, Pacific. J. Math. 219 (2005), 285--302.
\bibitem[JMX]{JMX} N. Jing, K. C. Misra, C. Xu, {\em Bosonic realization of toroidal Lie algebras of classical types}, Proc. AMS. 137 (2009), 3609--3618.
 \bibitem[JZ]{JZ} N. Jing,  H. Zhang, {\em Fermionic realization of two-parameter quantum affine algebra $U_{r,s}({sl_n})$}, Lett. Math. Phys. 89 (2009), no. 2, 159--170.
\bibitem[K1]{K1} V. G. Kac, {\em Lie superalgebras}, Adv. Math. 26 no. 1 (1977), 8--96.
\bibitem[K2]{K2} V. G.  Kac, {\em Vertex algebras for beginners}, Univ. Lect. Ser.10, AMS, 1997.
\bibitem[KP]{KP} V. G. Kac, D. Peterson, {\em Spin and wedge representations of infinite-dimensional Lie algebras and groups},
Proc. Nat. Acad. Sci. U.S.A. 78 (1981), no. 6, 3308--3312.
\bibitem[KW]{KW} V. G. Kac, M. Wakimoto, {\em Integrable highest
weight modules over affine superalgebras and Appell's function},
Commun. Math. Phys. 215 (2001), 631--682.
\bibitem[L]{L} M. Lau, {\em Bosonic and Fermionic represntation of Lie algebra with central extensions}, Adv. Math. 194(2) (2005), 225--245.
\bibitem[MRY]{MRY} R. V. Moody, S. E. Rao, T. Yokonuma, {\em Toroidal Lie algebras and vertex representations}, Geom. Ded. 35 (1990), 283--307.
\bibitem[T]{T} S. Tan, {\em Vertex operator representations for toroidal Lie algebra of type $B_l$}, Comm. Algebra 27 (1999), no. 8, 3593--3618.
\bibitem[X]{X} X. Xu, {\em Introduction to vertex operator superalgebras and their modules}, Kluwer Academic Publishers, Dordrecht, 1998.
\end{thebibliography}
\end{document}